\documentclass[11pt,oneside]{amsart}
\usepackage{fullpage}
\usepackage{amsthm, amsmath, amssymb, amsfonts, url, xcolor, setspace, fancyhdr}
\usepackage{bm} 
\usepackage{enumitem}
\usepackage[margin=22mm,bottom=14mm,footskip=7mm,top=13mm]{geometry}
\usepackage{tikz-cd} 

\usepackage[hidelinks]{hyperref}
\usepackage[capitalise]{cleveref}

\newtheorem{theorem}{Theorem}[section]
\newtheorem{proposition}[theorem]{Proposition}
\newtheorem{lemma}[theorem]{Lemma}

\newtheorem{observation}[theorem]{Observation}
\newtheorem{conjecture}[theorem]{Conjecture}
\newtheorem{claim}[theorem]{Claim}
\newtheorem{question}[theorem]{Question}

\theoremstyle{definition}

\theoremstyle{definition}
\newtheorem*{alg}{The algorithm}

\theoremstyle{remark}

\newtheorem*{remarks}{Remarks}

\newcommand{\eps}{\varepsilon}

\DeclareMathOperator{\bp}{bp}
\DeclareMathOperator{\VC}{VC}
\DeclareMathOperator{\diam}{diam}

\newcommand\restr[2]{{
  \left.\kern-\nulldelimiterspace 
  #1 
  \right|_{#2} 
  }}

\newcommand{\mF}{\mathcal{F}}
\newcommand{\mH}{\mathcal{H}}
\newcommand{\mC}{\mathcal{C}}

\newcommand{\topo}{\mathrm{top}}
\newcommand{\NN}{\mathbb{N}}
\newcommand{\RR}{\mathbb{R}}
\newcommand{\mV}{\mathcal{V}}
\newcommand{\mU}{\mathcal{U}}
\newcommand{\mW}{\mathcal{W}}
\newcommand{\mG}{\mathcal{G}}

\newcommand{\mB}{\mathcal{B}}

\begin{document}
\title{Variants of VC dimension and their applications to dynamics}

\author[Guorong Gao]{Guorong Gao}
\address{School of Mathematics and Statistics, Fuzhou University, Fuzhou, China}
\email{grgao@fzu.edu.cn}

\author[Jie Ma]{Jie Ma}
\address{School of Mathematical Sciences, University of Science and Technology of China, Hefei, China, and Yau Mathematical Sciences Center, Tsinghua University, Beijing, China}
\email{jiema@ustc.edu.cn}

\author[Mingyuan Rong]{Mingyuan Rong}
\address{School of Mathematical Sciences, University of Sciences and Technology of China, Hefei,
China}
\email{rong\_ming\_yuan@mail.ustc.edu.cn}
\author[Tuan Tran]{Tuan Tran}
\address{School of Mathematical Sciences, University of Sciences and Technology of China, Hefei,
China.}
\email{trantuan@ustc.edu.cn}

\date{}
\begin{abstract}
{Since its introduction by Vapnik and Chervonenkis in the 1960s, the VC dimension and its variants have played a central role in numerous fields. In this paper, we investigate several variants of the VC dimension and their applications to dynamical systems. First, we prove a new bound for a recently introduced generalization of VC dimension, which unifies and extends various extremal results on the VC, Natarajan, and Steele dimensions. This new bound allows us to strengthen one of the main theorems of Huang and Ye [Adv. Math., 2009] in dynamical systems. Second, we refine a key lemma of Huang and Ye related to a variant of VC dimension by providing a more concise and conceptual proof. We also highlight a surprising connection among this result, combinatorics, dynamical systems, and recent advances in communication complexity.}

\end{abstract}

\maketitle

\section{Introduction}

The Vapnik--Chervonenkis dimension \cite{VC68,VC71}, or VC dimension, is a combinatorial parameter of significant importance in various fields, including statistical learning theory \cite{VC71,BEHW89}, probability \cite{Ver18}, functional analysis \cite{Tal94}, discrete and computational geometry \cite{CW89,HW87}, model theory \cite{She72} and combinatorics \cite{DSW94}. The VC dimension of a family of binary vectors $\mH\subseteq \{0,1\}^n$ is the maximum size of a set \textbf{shattered} by the family, i.e., a set $S\subseteq [n]$ such that the projection of $\mH$ onto the coordinates of $S$ equals $\{0, 1\}^S$. A cornerstone result in this area is the Sauer–Shelah Lemma \cite{Sau72,She72,VC71}, which states that any family $\mH\subseteq \{0,1\}^n$ with VC dimension $d$ satisfies $|\mH|\le \sum_{i=0}^d\binom{n}{i}$. This bound is tight, as exemplified by the family of all binary vectors of length $n$ with at most $d$ ones. Various proofs of the Sauer-Shelah lemma can be found in the literature, and numerous variants of the lemma are known (see, e.g., \cite{Ste78, FS12, DSS18, Nat89, AHHM21}).

In this paper, we study a generalization of the VC dimension recently introduced in computer science and game theory, a variant of the VC dimension applied to partial concept classes, and their implications in dynamical systems.

\subsection{A generalization of VC dimension: \texorpdfstring{$k$}{k}-Natarajan dimension}

The concept of \textbf{$k$-Natarajan dimension} was introduced recently and has found applications in computer science, game theory \cite{DSS18}, and machine learning \cite{CP23}. 
For $r \geq k \geq 2$ and $n \geq 1$, the {$k$-Natarajan dimension} of a family $\mH \subseteq \{1, \ldots, r\}^{[n]}$, denoted by $\dim_k(\mH)$, is defined as the maximum size of a subset $S \subseteq [n]$ such that the projection of $\mH$ onto $S$ contains a subfamily of the form $\prod_{i \in S} Y_i$, where each $Y_i$ is a $k$-element subset of $\{1, \ldots, r\}$. When $r = k = 2$, this definition coincides with the VC dimension. For $k = 2$, it corresponds to the Natarajan dimension, and when $r = k$, it equals the Steele dimension. Our first result extends the Sauer--Shelah Lemma from the VC dimension to $k$-Natarajan dimension.

\begin{theorem}\label{thm:Natarajan}
Let $r\ge k \ge 2$ and $n\ge d\ge 1$. For any family $\mH\subseteq \{1,\ldots,r\}^{[n]}$ with $\dim_k(\mH) \le d$,
\begin{equation*}\label{eq:size}
|\mH|\le (k-1)^{n-d}\sum_{i=0}^d \binom{n-i-1}{d-i}\binom{r}{k}^{d-i}r^{i}.
\end{equation*}
\end{theorem}

\begin{remarks}\textcolor{white}{}
\begin{itemize}[topsep=0pt,parsep=0pt,partopsep=0pt]
     \item[(i)] \cref{thm:Natarajan} gives $|\mH|=O_{r,k,d}(n^d(k-1)^n)$. This bound is asymptotically tight, as shown by the family consisting of all vectors in which at most $d$ coordinates are greater than $k-1$. Moreover, the bound is sharp whenever $r=k\ge 2$ (see (iii) and (v) below).
     \item[(ii)] Weaker bounds were obtained by Daniely, Schapira and Shahaf \cite[Theorem 1.5]{DSS18} and Charikar and Pabbaraju \cite[Theorem 7]{CP23} using a different method.

    \item[(iii)] \cref{thm:Natarajan} recovers the classic Sauer--Shelah Lemma \cite{VC71,Sau72,She72} by taking $r=k=2$ and noting that
    $\sum_{i=0}^d\binom{n-i-1}{d-i}2^i=\sum_{i=0}^d\binom{n}{i}$.
    \item[(iv)] Specializing \cref{thm:Natarajan} to $k = 2$ yields an improved version of the Natarajan theorem \cite{Nat89}.
    \item[(v)] For $r=k$,  
    \cref{thm:Natarajan} gives the same bound as in the Steele theorem \cite{Ste78}.
\end{itemize}
\end{remarks}

Our novel approach in the proof of \cref{thm:Natarajan} is to construct a bijection between $\mH$ and a classification table that captures the shattering relationship. This idea, inspired by \cite{FS12}, might be of independent interest. Clearly, \cref{thm:Natarajan} has broad applicability, since it unifies and enhances various extremal results on the VC dimension. As expected, in \cref{thm:dim0}, we apply \cref{thm:Natarajan} to dynamical systems, improving one of the main theorems of Huang and Ye \cite{HY09}. To the best of our knowledge, this is the first application of the $k$-Natarajan dimension to topological dynamics.

\subsection{A variant of VC dimension to partial concept classes}

Given integers $r\ge 2$ and $n\ge 1$, we consider partial concept classes $\mH \subseteq \{1,\ldots,r,\star\}^{[n]}$, where each $h\in \mH$ is a partial vector; specifically if $i \in [n]$ is such that $h(i)=\star$ then $h$ is undefined at $i$.
The (non-traditional) {\bf VC dimension} of $\mH$, denoted $\dim_{\VC}(\mH)$, is the maximum size of a shattered set $S\subseteq Z$, where $S$ is said to be {\bf shattered} if  
the projection of $\mH$ onto $S$ contains $\{1,\ldots,r\}^S$. A 
family $\mF \subseteq \{1,\ldots,r\}^{[n]}$ is called a {\bf net} of $\mH$ if for every $h\in \mH$ there exists $f\in \mF$ such that
$h(i)\in \{1,\ldots,r,\star\}\setminus\{f(i)\}$ for all
$i\in [n]$. The smallest possible size of a net of $\mH$ is the {\bf covering number} of $\mH$, denoted $C(\mH)$. 

To investigate maximal pattern entropy of a topological dynamical system (the notations for which will be introduced in \cref{susec:topo}), Huang and Ye \cite{HY09} provide a crucial lemma showing that for sufficiently large $n$, if $\mH\subseteq \{1,\ldots,r,\star\}^{[n]}$ is a partial concept class with $\VC(\mH) \le d$, then
\[
C(\mH) \le r^2 2^m\left(\frac{n}{m}\right)^{2m},
\]
where $m:=\log_{\tfrac{r+1}{r}} \binom{n}{\le d}+1$ and $\binom{n}{\le d}:=\binom{n}{0}+\binom{n}{1}+\ldots+\binom{n}{d}$.
In the following theorem, we strengthen this result and provide a more concise and conceptual proof.
\begin{theorem}
\label{thm:comb}
Let $r \ge 2$ and $n\ge d\ge 1$. If $\mH \subseteq {\{1,\ldots,r,\star\}}^{[n]}$ is a partial concept class with $\VC(\mH) \le d$, then
\begin{equation}\label{eq:covering number}
C(\mH) \le \binom{n}{\le\log_{\tfrac{r}{r-1}} \binom{n}{\le d}}.
\end{equation}
\end{theorem}

Earlier versions of this lemma (see \cite{VC71,Sau72,She72,KM78,KL07}) play a significant role in all aspects of local entropy theory for topological dynamical systems. For a thorough discussion we refer the reader to the survey by Glasner and Ye \cite{GY09}, and Chapter 12 of the book by Kerr and Li \cite{KL16}.
The case $r=2$ of \cref{thm:comb} is a result of Alon, Hanneke, Holzman and Moran \cite[Theorem 12]{AHHM21} in their study of PAC learning theory. Our proof extends their ideas.
    
When $r=2$ and $d=1$, the inequality \eqref{eq:covering number} gives $C(\mH)\leq \binom{n}{\le\log (n+1)}\leq n^{\log( n+1)}$. More generally, one can show that the 
right-hand side of the inequality \eqref{eq:covering number} is at most $n^{(r-1)d\log (n+1)}$. Note also that for $r=O(1)$ and $d=o(n)$, we have $\log_{\tfrac{r}{r-1}} \binom{n}{\le d}=o(n)$ since $\binom{n}{\le o(n)}=2^{o(n)}$, which implies $C(\mathcal{H})\le \binom{n}{\le o(n)}=2^{o(n)}$. 
\vskip 0.3cm
Another result of Alon, Hanneke, Holzman and Moran \cite[Theorem 11]{AHHM21} shows that the inequality \eqref{eq:covering number} is nearly tight when $r=2$ and $d=O(1)$. Interestingly, its proof hinges on a recent breakthrough in communication complexity and its implications in graph theory by Balodis, Ben-David, G\"o\"os, Jain and Kothari \cite{BBGJK21}.

\begin{theorem}[\cite{AHHM21}]
\label{thm:lower-bound}
There is a partial concept class $\mH\subseteq\{1,2,\star\}^{[n]}$ with $\VC(\mH)\le 1$ and $C(\mH) \ge n^{(\log  n)^{1-o(1)}}$.
\end{theorem}

Since \cref{thm:lower-bound} was stated in a different form in \cite{AHHM21}, 
we include its proof in \cref{sec:lower-bound} for completeness.

\subsection{Applications to dynamical systems}\label{susec:topo}
The central object of study in topological dynamics is a topological dynamical system (TDS) $(X,T)$, where $X$ is a nonempty compact metrisable space and $T\colon X\rightarrow X$ is a continuous map. Ever since Adler, Konheim and McAndrew \cite{AKM65} introduced entropy into dynamical systems theory 60 years ago, it has played a very central role in the study of topological dynamical systems (see, for example, the surveys \cite{HK02,Kat23}). They associated to any topological dynamical system $(X,T)$ a topological invariant $h_{\topo}(T) \in \RR_{+} \cup \{\infty\}$, called the {\bf topological entropy} of $(X,T)$, which measures the uncertainty or disorder of the system. Systems with positive topological entropy are random in certain sense, and systems with zero topological entropy are said to be {\bf deterministic} even though they may exhibit complicated behaviours. Comparing to positive entropy systems, deterministic systems are much less understood. In order to distinguish between deterministic systems, Huang and Ye \cite{HY09} introduced the concept of maximal pattern entropy $h^*_{\topo}(T)$ of a topological dynamical system $(X,T)$. It is known that both $h_{\topo}(T)$ and $h^*_{\topo}(T)$ take value in $\{0,\log 2, \log 3, \ldots\} \cup \{\infty\}$, and that $h_{\topo}(T)>0$ implies $h^*_{\topo}(T)=\infty$ (see \cite{AKM65,HY09}). Hence maximal pattern entropy is especially useful for deterministic systems.

We now proceed to give more details.
Let $(X,T)$ be a TDS. Given two finite open covers $\mU,\mV$ (of $X$), their {\bf joint} is defined as
$\mU\vee\mV:=\{A \cap B\colon A\in \mU, B\in \mV\}$.
Clearly, $\mU\vee\mV$ is also an open cover of size at most $|\mU||\mV|$, and at least $1$. For a finite open cover $\mU$, let $p^*_{X,\,\mU}\colon \NN \rightarrow \NN$ be the function given by
\[
p^*_{X,\,\mU}(n)=\max_{S \subset \NN \cup \{0\},\, |S|=n} N\big(\bigvee_{i\in S} T^{-i} \mU\big),
\]
where $N(\mV)$ denotes the minimum size of a subcover chosen from an open cover $\mV$.
It is easy to see that the sequence $\{\log p^*_{X,\,\mU}(n)\}_{n\in \NN}$ is sub-additive, and thus by Fekete's lemma, the limit $
\lim\limits_{n \rightarrow+\infty} \frac{1}{n} \log p^*_{X,\, \mU}(n)$ exists.  Denote this limit by $h_{\topo}^*(T, \mU)$. The {\bf maximal pattern entropy} of $(X, T)$ is then defined as
\[
h_{\topo}^*(T):=\sup _{\mU}h_{\topo}^*(T, \mU),
\]
where the supremum is over all finite open covers $\mU$.\footnote{This definition is independent of a choice of metric.}

We say a topological dynamical system $(X,T)$ is {\bf null} if $h_{\topo}^*(T)=0$, that is, $h^*_{\topo}(T)$ attains the minimum possible
value. From the definition of $h^*_{\topo}(T)$, we find that $(X,T)$ is null if and only if $p^*_{X,\,\mU}(n)$ grows sub-exponentially in $n$ for each finite open cover $\mU$. For such systems, an intriguing conjecture of Huang and Ye \cite{HY09} further rules out the intermediate growth between polynomial and exponential.

\begin{conjecture}[\cite{HY09}]
\label{conj:null}
If $(X,T)$ is a null TDS, then $p^*_{X,\,\mU}$ is of polynomial order for each finite open cover $\mU$.
\end{conjecture}

The conjecture was repeated in the survey on local entropy theory by Glasner and Ye \cite{GY09}. It was shown to be true for interval maps by Li \cite{Li11}, circle maps by Yang \cite{Yan21}, and most interestingly for zero-dimensional systems by Huang and Ye \cite{HY09}. A refinement of an open cover $\mU$ of $X$ is a new open cover $\mV$ of $X$ such that every set in $\mV$ is contained in some set in $\mU$. We say $X$ is {\bf zero-dimensional} if every finite open cover of $X$ has a clopen (closed and open) refinement. Using the Karpovsky--Milman Theorem \cite{KM78}, which is a generalization of the Sauer--Shelah Lemma, Huang and Ye \cite[Theorem 5.4]{HY09} verified \cref{conj:null} for such spaces $X$.

\begin{theorem}[\cite{HY09}]
\label{thm:HY-dim0}
Let $(X, T)$ be a TDS and let $\mU=\left\{U_1, \ldots, U_r\right\}$ be a clopen partition of $X$. Then $h_{\topo}^*(T, \mU)=\log \ell$ for some $\ell \in \{1,\ldots,r\}$. Moreover, one of the following alternatives holds.
\begin{itemize}
    \item[(a)] If $\ell=r$, then $p_{X,\, \mU}^*(n)=r^n$ for all $n \in \mathbb{N}$.
    \item[(b)] If $\ell= r-1$, then there exists $c>0$ such that $\ell^n \le p_{X,\, \mU}^*(n) \le n^c \ell^n$ for all $n \ge 2$.
    \item[(c)] If $2\le \ell \leq r-2$, then there exists $c>0$ such that $\ell^n \le p_{X,\, \mU}^*(n) \le  n^{c\log n} \ell^n$ for all $n \ge 2$.
    \item[(d)] If $\ell=1$, then there exists $c>0$ such that $ p_{X,\, \mU}^*(n) \le n^c$ for all $n \ge 2$.
\end{itemize}
\end{theorem}

Alternative (d) tells us that \cref{conj:null} does hold when $X$ is zero-dimensional. Indeed, let $\mU$ be any finite open cover of $X$. Then $\mU$ has a clopen refinement $\mV=\{V_1,\ldots,V_r\}$, which forms a partition of $X$. Since $\mV$ is a refinement of $\mU$, $p^*_{X,\,\mU}(n)\le p^*_{X,\,\mV}(n)$ for all $n$. Moreover, as $(X,T)$ is a null TDS, $h_{\topo}^*(T, \mU)=0$, and so it follows from \cref{thm:HY-dim0} (d) that
$p^*_{X,\, \mV}(n) \le n^{c}$ for $n\ge 2$. Therefore, $p^*_{X,\,\mU}(n)\le n^{c}$ for all $n\ge 2$, as desired.

Huang and Ye \cite{HY09} handled alternatives (b)--(d) separately.
Using our combinatorial lemma (i.e., \cref{thm:Natarajan}), we improve the multiplicative factor in (c) from quasipolynomial to polynomial in $n$, and provide a unified treatment of (b)--(d) as follows.

\begin{theorem}\label{thm:dim0}
Let $(X, T)$ be a TDS, and let $\mU=\left\{U_1, \ldots, U_r\right\}$ be a clopen partition of $X$. Then $h_{\topo}^*(T, \mU)=\log \ell$ for some $\ell \in \{1,\ldots,r\}$. Moreover, one of the following alternatives holds:
\begin{itemize}
  \item[(a)] If $\ell=r$, then $p_{X,\, \mU}^*(n)=r^n$ for all $n \in \mathbb{N}$.
  \item[(b)] If $\ell\le r-1$, then there exists $c>0$ such that $\ell^n \le p_{X,\, \mU}^*(n) \le n^c \ell^n$ for all $n \ge 2$.

\end{itemize}
\end{theorem}

For general topological spaces, the following result, due to Huang and Ye \cite[a special case of Theorem 4.5]{HY09},  
represents the current state of the art of \cref{conj:null}.

\begin{theorem}[\cite{HY09}]\label{thm:HY-general}
If $(X,T)$ is a null TDS, then for each finite open cover $\mU$ there is a constant $c=c(\mU)$ such
that $p^{*}_{X,\,\mU}(n) \le n^{c \log n}$ for every $n \ge 2$.
\end{theorem}

At the heart of Huang and Ye’s argument in \cref{thm:HY-general} is 
an earlier version of \cref{thm:comb}.
To illustrate the application of \cref{thm:comb} in dynamical systems, we provide a (simplified) proof of \cref{thm:HY-general} in \cref{sec:topo}. 

\subsection{Notation and organization}
 We adopt standard notation throughout. In particular, $[n]$ denotes the set $\{1,2,\ldots,n\}$. 
 Given a set $S$ and $k \in \mathbb{N}$, we write $\binom{S}{k}$ for the collection of all $k$-subsets of $S$.
 The notation $Y^{Z}$ refers to the family of all functions (or vectors) from a set $Z$ to a set $Y$. For $h\in Y^Z$ and $S\subseteq Z$, the projection of $h$ onto $S$, denoted $\restr{h}{S}$,  is the map in $Y^S$ defined by $i\mapsto h(i)$. Similarly, for $\mH \subseteq Y^Z$ and $S\subseteq Z$, the projection of $\mH$ onto $S$ is given by
\[
\restr{\mH}{S}=\{\restr{h}{S}: h\in \mH\}.
\]
For two nonnegative functions $f$ and $g$ of some underlying parameter $n$, we write $f=O(g)$ if there exist positive constants $C$ and $n_0$ such that $f(n)\le Cg(n)$ for all $n\ge n_0$, and $f=o(g)$ if $\lim_{n\rightarrow \infty}f(n)/g(n)=0$. To simplify the presentation, we systematically omit floor and ceiling signs. Unless specified otherwise, all logarithms are taken to base $2$.

We give the proof of \cref{thm:Natarajan} in \cref{sec:Natarajan} and the proof of \cref{thm:comb} in \cref{sec:partial classes}.
We then derive \cref{thm:dim0} from \cref{thm:Natarajan} in \cref{sec:dim0}. In \cref{sec:symbolic} we use \cref{thm:dim0} to study the complexity of symbolic dynamics (see \cref{thm:symbolic}). Finally, we close the paper with some concluding remarks in \cref{sec:conclusion}.

\section{\texorpdfstring{$k$}{}-Natarajan dimension}\label{sec:Natarajan}

In this section we provide a proof of \cref{thm:Natarajan}. 
Throughout the section, we view elements of product spaces as vectors. For a vector $x$, we denote by $x_i$ the value of the $i$-th coordinate. For integers $a$ and $b$, we employ the interval notation
\[
[a,b]:=\{x\in \mathbb{Z}\colon a \le x \le b\}.
\]
The following lemma makes up the bulk of the proof of \cref{thm:Natarajan}.
\begin{lemma}\label{lem:branching}
Given integers $r$ and $k$ with $r\ge k\ge 2$, let $\Omega$ denote the alphabet
\[
\{b_1,b_2,\dots,b_{k-1}\}\cup\Big\{c_A\colon A\in \binom{[r]}{k}\Big\}
\]
of size $k-1+\binom{r}{k}$. Then for every family $\mH \subseteq [r]^{[n]}$, there exist $n$ maps $\varphi_1,\ldots,\varphi_n$ and $n+1$ families $\mH_0:=\mH, \mH_1,\ldots,\mH_n$ (see the diagram below) with the following properties.
\begin{itemize}
    \item[(P1)] For $0\le i\le n$, $\mH_i$ is a subfamily of $\Omega^{[i]}\times {[r]}^{[i+1,n]}$.
    \item[(P2)] For $1\le i\le n$, $\varphi_i\colon \mH_{i-1}\rightarrow \mH_i$ is a bijection that preserves all but the $i$-th coordinate.
    \item[(P3)] Every element in $\mH_n$ has at most $\dim_k(\mH)$ coordinates in $\{c_A\colon A\in \binom{[r]}{k}\}$.
\end{itemize}
\end{lemma}

\[
\begin{tikzcd}
{[r]}^{[n]} & \Omega^{[1]}\times {[r]}^{[2,n]} & \Omega^{[2]}\times {[r]}^{[3,n]} & \cdots & \Omega^{[n]}\\
\mH_0=\mH\arrow[r,"\varphi_1"]\arrow[u, phantom, sloped, "\subseteq"] & \mH_1\arrow[r,"\varphi_2"]\arrow[u, phantom, sloped, "\subseteq"] & \mH_2\arrow[r]\arrow[u, phantom, sloped, "\subseteq"] & \cdots\arrow[r,"\varphi_n"]& \mH_n\arrow[u, phantom, sloped, "\subseteq"]
\end{tikzcd}
\]

\vspace{4mm}

Assuming \cref{lem:branching}, we now give a proof of \cref{thm:Natarajan}.

\begin{proof}[Proof of \cref{thm:Natarajan} assuming \cref{lem:branching}]
Consider a family $\mH\subseteq [r]^{[n]}$ with $\dim_k(\mH) \le d$. Applying \cref{lem:branching} to $\mH$, we get $\Omega,\varphi_1,\ldots,\varphi_n,\mH_1,\ldots,\mH_n$. Set $\Omega'=\{b_1,\ldots,b_{k-1}\}, \Omega''=\{c_A\colon A\in \binom{[r]}{k}\}$ and $\varphi=\varphi_n\circ \cdots \circ \varphi_1$.
By (P3), for every $x\in \mH$, $\varphi(x)$ has at least $n-d$ coordinates in $\Omega'$. Hence $\mH$ is partitioned into $d+1$ subfamilies $\mG_0,\ldots,\mG_d$, where
\[
\mathcal{G}_i=\{x\in \mH\colon \varphi(x)_{n-i}\in \Omega' \text{ and } \varphi(x)|_{[n-i]} \text{ has exactly } n-d \text{ coordinates in } \Omega'\}.
\]
Denote
\[
\mathcal{F}_i=\{(\varphi_{n-i}\circ \cdots \circ \varphi_1)(x)\colon x\in \mathcal{G}_i\}.
\]
From (P1) and (P2), we get $\mathcal{F}_i\subseteq \Omega^{[n-i]}\times [r]^{[n-i+1,n]}$ and $|\mathcal{F}_i|=|\mathcal{G}_i|$.
To bound the size of $\mF_i$, let $y$ be any vector of $\mF_i$. Then $y=(\varphi_{n-i}\circ \cdots \circ \varphi_1)(x)$ for some $x\in \mG_i$. By (P2), $\varphi(x)=(\varphi_n\circ \cdots\circ \varphi_{n-i+1})(y)$ and $y$ agree in the first $n-i$ coordinates. On the other hand, as $x\in \mG_i$, we know that $\varphi(x)_{n-i}\in \Omega'$, and that $\varphi(x)|_{[n-i]}$ has exactly $n-d$ (respectively $d-i$) coordinates in $\Omega'$ (respectively $\Omega''$). Therefore, $y_{n-i}\in \Omega'$, and $y|_{[n-i]}$ has exactly $n-d$ (respectively $d-i$) coordinates in $\Omega'$ (respectively $\Omega''$). From this, we obtain
\[
|\mathcal{F}_i|_{[n-i]}|\leq \binom{n-i-1}{n-d-1}(k-1)^{n-d}\binom{r}{k}^{d-i}=(k-1)^{n-d}\binom{n-i-1}{d-i}\binom{r}{k}^{d-i}.
\]
Combining this with the trivial bound $|\mathcal{F}_i|_{[n-i+1,n]}|\leq r^i$ yields
\[
|\mathcal{F}_i|\le |\mathcal{F}_i|_{[n-i]}|\cdot |\mathcal{F}_i|_{[n-i+1,n]}|\le (k-1)^{n-d}\binom{n-i-1}{d-i}\binom{r}{k}^{d-i} \cdot r^{i}.
\]
Therefore, we have
$$|\mH|=\sum_{i=0}^d|\mathcal{G}_{i}|=\sum_{i=0}^d|\mathcal{F}_{i}|\leq (k-1)^{n-d}\sum_{i=0}^d \binom{n-i-1}{d-i}\binom{r}{k}^{d-i}r^{i},$$
finishing the proof.
\end{proof}

\paragraph{\bf Construction}
We recursively define $\varphi_1,\mH_1,\ldots, \varphi_n,\mH_n$. Suppose we have defined $$\mH_0,\varphi_1,\mH_1,\ldots,\varphi_{i-1},\mH_{i-1}$$ for some $1\leq i \leq n$. We shall construct a map $\varphi_i\colon \mH_{i-1}\rightarrow \Omega^{[i]}\times [r]^{[i+1,n]}$ and let $\mH_i=\varphi_i(\mH_{i-1})$.
For $x\in \mH_{i-1}$, the $i$-th {\em block} of $\mH_{i-1}$ containing $x$ is the family
\[
\partial^{(i)}(x)=\{y\in \mH_{i-1}\colon \text{$x$ and $y$ agree in all but possibly the $i$-th coordinate}\}.
\]
Note that $x\in \partial^{(i)}(x)$. Since $y\in \partial^{(i)}(x)$ if and only if $\partial^{(i)}(x)=\partial^{(i)}(y)$, $\mH_{i-1}$ is decomposed into $i$-th blocks. We thus only need to define $\varphi_i$ on each $i$-th block. Given $x=(x_1,\ldots,x_n)\in \mH_{i-1}$, order the elements of $\partial^{(i)}(x)$ as
\[
(x_1,\ldots,x_{i-1},s_1,x_{i+1},\ldots,x_n),\ldots, (x_1,\ldots,x_{i-1},s_t,x_{i+1},\ldots,x_n), \text{with $1\le s_1<\ldots<s_t \le r$}.
\]
For $1 \le j \le t$, define $\varphi_i(x_1,\ldots,x_{i-1},s_j,x_{i+1},\ldots,x_n)$ to be the vector
\begin{itemize}
\item Except the $i$-th coordinate, the other coordinates of $\varphi_i(x_1,\ldots,x_{i-1},s_j,x_{i+1},\ldots,x_n)$ and $x$ are the same;
\item The $i$-th coordinate of $\varphi_i(x_1,\ldots,x_{i-1},s_j,x_{i+1},\ldots,x_n)$ is $b_j$ if $1\le j\le k-1$, and $c_{\{s_1,\ldots,s_{k-1}, s_j\}}$ if $j\ge k$.
\end{itemize}
As these vectors $\varphi_i(x_1,\ldots,x_{i-1},s_j,x_{i+1},\ldots,x_n)$ lie in $\Omega^{[i]}\times {[r]}^{[i+1,n]}$, $\mH_i=\varphi_i(\mH_{i-1})$ is a subfamily of $\Omega^{[i]}\times {[r]}^{[i+1,n]}$, thereby verifying (P1).

To verify the other properties, we first make some simple observations that follow readily from the definition of $\varphi_i$.

\begin{observation}\label{obs:branching} The following hold for every $i\in [n]$ and $x\in \mH_{i-1}$.
\begin{itemize}
    \item[(a)] $\varphi_i(x)$ preserves all but the $i$-th coordinate of $x$.
    \item[(b)] The restriction of $\varphi_i$ to $\partial^{(i)}(x)$ is injective.
    \item[(c)] Suppose $x=(x_1,\ldots,x_n)\in \mH$ and the $i$-th coordinate of $\varphi_i(x)$ is $c_A$ for some $A\in \binom{[r]}{k}$. Then $\mH_{i-1}\supseteq \partial^{(i)}(x)$ contains $\{x_1\}\times \cdots \times \{x_{i-1}\}\times A \times \{x_{i+1}\} \times \cdots \times \{x_n\}$. 
\end{itemize}
\end{observation}

\begin{proof}[Proof of \cref{lem:branching} (continuation)]
(P2) As $\mH_i=\varphi_i(\mH_{i-1})$, evidently $\varphi_i$ is surjective. So what's left is to show that $\varphi_i$ is injective. Suppose $\varphi_i(x)=\varphi_i(y)$.
By \cref{obs:branching} (a), $\varphi_i$ preserves all but the $i$-th coordinate, so $x$ and $y$ agree in all but possibly the $i$-th coordinate. Thus, $x$ and $y$ are two vectors of $\partial^{(i)}(x)$ with $\varphi_i(x)=\varphi_i(y)$. But by \cref{obs:branching} (b), the map $\varphi_i$ restricted to $\partial^{(i)}(x)$ is injective, so one must have $x=y$. This proves (P2).

(P3) Let $d=\dim_k(\mH)$. Suppose for the contrary that there is a vector $y\in \mH_n$ together with a $(d+1)$-element subset $I \subseteq [n]$ such that for every $i\in I$, $y_i=c_{A_i}$ for some $A_i\in \binom{[r]}{k}$. By repeatedly applying \cref{obs:branching} (a) and (c) in a reversed ordering with respect to the coordinates, we find that $\restr{\mH}{I}$ contains $\prod_{i\in I} A_i$, implying $\dim_k(\mH)\geq d+1$, a contradiction. This finishes the proof.
\end{proof}

\section{A variant of VC dimension to partial concept classes}\label{sec:partial classes}

In this section, we provide a short proof of \cref{thm:comb}.
We need to show that any partial concept class with bounded $\VC$-dimension admits a small net. We shall construct the net via an algorithm. We first fix some notations and make some innocuous observations.

Let $\mH \subseteq \{1,\ldots,r,\star\}^{Z}$ be a partial concept class. The {\bf shattering strength} of $\mH$, denoted by $s(\mH)$, is the number of subsets $S\subseteq Z$ that are shattered by $\mH$. By convention, the shattering strength of the empty class is $0$, and the empty set is shattered by all nonempty classes (and so the shattering strength of any nonempty class is at least $1$). It is easy to see that $s(\mH) \le \binom{|Z|}{\le \VC(\mH)}$. For $(i,j) \in Z\times [r]$, we denote
\[
\mH_{i\rightarrow j}=\left\{h \in \mH\colon h(i)=j\right\}.
\]
Define the VC-minority function $M_{\mH}\colon Z\rightarrow [r]$ of $\mH$ by letting $M_{\mH}(i)$ be the value $j \in [r]$ which minimises $s(\mH_{i\rightarrow j})$, with an arbitrary tie-breaking rule. Observe that for any $i\in Z$,
\begin{equation}\label{ineq:shrinking-1}
(r-1)\cdot s(\mH) \ge s(\mH_{i\rightarrow 1})+\ldots + s(\mH_{i\rightarrow r}).
\end{equation}
In particular,
\begin{equation}\label{ineq:shrinking-2}
s(\mH_{i\rightarrow j}) \le \tfrac{r-1}{r}\cdot s(\mH), \text{ where } j=M_{\mH}(i).
\end{equation}
To see \eqref{ineq:shrinking-1}, for any subset $S\subseteq Z$ with $i\notin S$, we consider the contribution of the pair $S, S\cup \{i\}$ to both sides of the inequality.
We note that every set $S$ that is shattered by one of the classes $\mH_{i\rightarrow 1},\ldots,\mH_{i\rightarrow r}$ is also shattered by $\mH$,
and if $S$ is shattered by all of the $\mH_{i\rightarrow j}$ then both $S$ and $S\cup \{i\}$ are shattered by $\mH$.

We shall use the following algorithm to construct a small net of a given partial concept class.

\begin{alg} 
Fix a partial concept class $\mH\subseteq \{1,\ldots,r,\star\}^{[n]}$. For any partial function $h\in \mH$, the algorithm will output an index set $A^{n}\subseteq [n]$ and a total function $f\in \{1,\ldots,r\}^{[n]}$. Set $A^{0}=\emptyset$ and $\mH^{0}=\mH$. For $i=1,\ldots,n$, do the following:

\noindent (1) Compute the value of the VC-minority function of $\mH^{i-1}$ at $i$. Denote this value by $j$.

\noindent (2) If $h(i)\ne j$, then set $f(i)=j$, $A^{i}=A^{i-1}$ and $ \mH^{i}=\mH^{i-1}$.

\noindent (3) If $h(i)=j$, then set $f(i)=j+1 \pmod{r}$, $A^{i}=A^{i-1}\cup\{i\}$ and $\mH^{i}=(\mH^{i-1})_{i\rightarrow j}$.
\end{alg}

The outputs of the algorithm satisfy the following properties.

\begin{lemma}\label{lem:algorithm}
\textcolor{white}{}
\begin{itemize}
\item[(a)] $h\in \mH^i$ for every $0\le i \le n$. In particular, $s(\mH^i)\ge 1$ for every $0\le i \le n$.
\item[(b)] $h(i) \ne f(i)$ for every $i\in [n]$.
\item[(c)] $f$ is determined by $A^{n}$.
\item[(d)] $A^{n}$ is a subset of $[n]$ of size at most $\log_{\frac{r}{r-1}}s(\mH)$.
\end{itemize}
\end{lemma}
\begin{proof}
Properties (a) and (b) are easy to verify. For the others, write $A^{n}=\{a_1<a_2<\cdots<a_k\}$ and let $a_0=0$, $a_{k+1}=n+1$. Given $i\in [n]$, there must exist $\ell=\ell(i)\in [k+1]$ with $a_{\ell-1}<i\le a_{\ell}$. Denote by $j=j(i)$ the value of the VC-minority function of $\mH^{a_{\ell-1}}$ at $i$. A simple induction on $\ell$ shows that $A^i=\{a_1,\ldots,a_{\ell-1}\}$,
$
\mH^i=\mH^{a_{\ell-1}}$ and $f(i)=j$ when $a_{\ell-1}<i<a_{\ell}$, while $A^i=\{a_1,\ldots,a_{\ell}\}$,
$\mH^i=(\mH^{a_{\ell-1}})_{i\rightarrow j}$ and $f(i)=j+1 \pmod{r}$ when $i=a_{\ell}$. From this we see that for every $i\in [n]$, $A^i,\mH^i$ and $f(i)$ are uniquely determined by $A^n$ and $i$.
This implies (c).

From \eqref{ineq:shrinking-2} and the above discussion, we get $s(\mH^{a_{\ell}}) \le \frac{r-1}{r} \cdot s(\mH^{a_{\ell-1}})$ for every $1\le \ell \le k$. Together with (a), this implies $1\le s(\mH^{a_{k}}) \le (\frac{r-1}{r})^k \cdot s(\mH^{a_0})=(\frac{r-1}{r})^{k} \cdot s(\mH)$, which results in $k \le \log_{\frac{r}{r-1}}s(\mH)$, as desired.
\end{proof}

From \cref{lem:algorithm}, we quickly deduce \cref{thm:comb}.

\begin{proof}[Proof of \cref{thm:comb}]
Let $\mH \subseteq {\{1,\ldots,r,\star\}}^{[n]}$ be a partial concept class with $\VC(\mH) \le d$. Run the algorithm for each $h\in \mH$, and let $\mF$ be the family consisting of all outputs $f$. By property (b) in \cref{lem:algorithm}, $\mF$ is a net of $\mH$. Finally, from properties (c) and (d) in \cref{lem:algorithm}, we find
\[
|\mF| \le \binom{n}{\le \log_{\frac{r}{r-1}} s(\mH)} \le \binom{n}{\le \log_{\frac{r}{r-1}} \binom{n}{\le d}},
\]
where the second inequality holds since $s(\mH) \le \binom{n}{\le \VC(\mH)}\le \binom{n}{\le d}$. This completes our proof.
\end{proof}

\section{Applications to dynamical systems}

As mentioned earlier, in this section we study the complexity function $p^*_{X,\, \mU}(n)$ when $\mU$ is a clopen partition of $X$.
We derive the main result of this section, namely \cref{thm:dim0}, from
\cref{thm:Natarajan}. In \cref{sec:symbolic}, we use \cref{thm:dim0} to study the complexity of symbolic dynamics, improving another result of Huang and Ye \cite{HY09}.

We begin with an auxiliary lemma that is used to prove the statements given later in the subsections.

\begin{lemma}\label{lem:clopen-family}
Let $(X,T)$ be a TDS and let $\mU=\{U_1,\ldots,U_r\}$ be a clopen partition of $X$. Then, for any subset $S \subseteq \NN\cup\{0\}$,
\[
N\big(\bigvee_{i\in S} T^{-i}\mU\big)=\# \Big\{f\in [r]^{S}: \bigcap_{i\in S}T^{-i}U_{f(i)} \neq \emptyset\Big\}.
\]
\end{lemma}
\begin{proof}
Since $\mU$ is a partition of $X$, the sets in the open cover $\bigvee_{i\in S} T^{-i} \mU$ are pairwise disjoint. Hence the minimum size of a subcover of $\bigvee_{i\in S} T^{-i} \mU$ is exactly the number of nonempty sets in $\bigvee_{i\in S} T^{-i} \mU$. Moreover, every set in $\bigvee_{i\in S} T^{-i} \mU$ is of the form $\bigcap_{i\in S} T^{-i}U_{f(i)}$ for some function $f\in \{1,\ldots,r\}^S$. Therefore, $N\big(\bigvee_{i\in S} T^{-i}\mU\big)=\# \left\{f\in [r]^{S}: \bigcap_{i\in S}T^{-i}U_{f(i)} \neq \emptyset\right\}$.
\end{proof}

\subsection{An application in topological dynamical systems}\label{sec:dim0}

Given a total class $\mH \subseteq \{1,\ldots,r\}^Z$ and a positive integer $k$, we say a subset $S\subseteq Z$ is {\it $k$-Natarajan shattered} by $\mH$ if $\restr{\mH}{S}$ contains a subclass of the form $\prod_{i\in S}Y_i$, where $Y_i$ is a $k$-element subset of $\{1,\ldots,r\}$ for each $i\in S$. Then, $\dim_k(\mH)$ equals the maximum size of a $k$-Natarajan shattered set.

\begin{proof}[Proof of \cref{thm:dim0} assuming \cref{thm:Natarajan}]
Let $\mU=\{U_1,\ldots,U_r\}$ be a clopen partition of $X$. For each $n\in \NN$, let $V_n$ be a set of $n$ nonnegative integers such that
$N\big(\bigvee_{i\in V_n} T^{-i} \mU\big)=p^*_{X,\, \mU}(n)$.
Define
\[
\mH_n=\Big\{h\in [r]^{V_n}\colon \bigcap_{i\in V_n}T^{-i}U_{h(i)} \ne \emptyset\Big\}.
\]
Then, by \cref{lem:clopen-family}, $N\big(\bigvee_{i\in V_n} T^{-i} \mU\big)=|\mH_n|$, resulting in
\[
p^*_{X,\, \mU}(n)=|\mH_n|.
\]

Let $\ell$ be the maximum integer such that $\limsup_{n\rightarrow \infty} \dim_{\ell}(\mH_n)=\infty$. Then we have $\ell\in \{1,\ldots,r\}$.
We consider the upper bound on $p^*_{X,\, \mU}(n)$ first.
If $\ell=r$, then clearly $p_{X,\, \mU}^*(n) \le |\mU|^n=r^n=\ell^n$.
Now suppose $\ell\le r-1$, then we have $\dim_{\ell+1}(\mH_n)=O(1)$.
Since $\mH_n\subseteq \{1,\ldots,r\}^{V_n}$, $|V_n|=n$ and $\dim_{\ell+1}(\mH_n)=O(1)$, we derive from \cref{thm:Natarajan} that $|\mH_n|\le n^{O(1)}\ell^n$, and so
\[
p^*_{X,\, \mU}(n)=|\mH_n| \le n^{O(1)}\ell^n \mbox{~~ when ~} \ell\le r-1.
\]

We proceed to lower bound $p^*_{X,\, \mU}(n)$. Because $\limsup_{n\rightarrow \infty} \dim_{\ell}(\mH_n)=\infty$, there exist a sequence $\{n(k)\}_{k\in \NN}$ of positive integers and a sequence $\{W_k\}_{k\in \NN}$ of sets satisfying
\begin{itemize}
    \item[(i)] $W_k$ is $\ell$-Natarajan shattered by $\mH_{n(k)}$.
    \item[(ii)] $W_k$ is a subset of $V_{n(k)}$ of size $k$;
\end{itemize}
For every $k\in \NN$, we have
\begin{align*}
p_{X,\, \mU}^*(k)\overset{(ii)}{\ge} N\big(\bigvee_{i\in W_k} T^{-i} \mU\big)&=\# \Big\{f\in [r]^{W_k}: \bigcap_{i\in W_k}T^{-i}U_{f(i)} \neq \emptyset\Big\}\\
&\ge |\restr{\mH_{n(k)}}{W_k}|\overset{(i)}{\ge}{\ell}^{|W_k|}\overset{(ii)}{=}{\ell}^k,
\end{align*}
where the first equality follows from \cref{lem:clopen-family}, and in the second inequality we used the definition of $\mH_{n(k)}$.
We therefore get $p_{X,\, \mU}^*(n)=\ell^n$ when $\ell=r$ and
\[
{\ell}^n \le p_{X,\, \mU}^*(n) \le n^{O(1)}{\ell}^n \quad \text{ when $\ell\le r-1$.}
\]
From this we find $h_{\topo}^*(T,\mU)= \limsup _{n \rightarrow+\infty} \frac{1}{n} \log p^*_{X,\, \mU}(n)=\log \ell$, where $\ell\in \{1,\ldots,r\}$.
This completes our proof.
\end{proof}

\subsection{An application in symbolic dynamics}\label{sec:symbolic}

Given an integer $r\ge 2$, we consider the product set $\Omega_r=\{1, \ldots, r\}^{\mathbb{N}}$. Topology on $\{1, \ldots, r\}$ is discrete, and $\Omega_r$ is endowed with the product topology. Since $\{1, \ldots, r\}$ is compact and metrisable, so is the product space $\Omega_r$. Concretely, one can equip $\Omega_r$ with the metric $d((x_n)_{n\in \NN},(y_n)_{n\in \NN})=\sum_{n\ge 1}2^{-n}\bm{1}_{x_n\neq y_n}$.
Let $T\colon \Omega_r \rightarrow \Omega_r$ be the shift $T(x_n)_{n\in \NN}:=(x_{n+1})_{n\in \NN}$. A {\bf subshift} is a closed $T$-invariant subset of $\Omega_r$.

Consider a subshift $X\subseteq \Omega_r$. We see that $(X,T)$ is a TDS. For $1\le i \le r$, $U_i$ denotes the clopen set $X\cap \{x\in \Omega_r\colon x_1=i\}$.
Then $\mU_0:=\{U_1,\ldots,U_r\}$ is a clopen partition of $X$. Hence we can define
\[
p_X^*(n):=p^*_{X,\,\mU_0}(n) \enskip \text{for $n\in \NN,$ and $h^*(X):=h^*_{\topo}(T,\mU_0)$}.
\]
As a direct application of \cref{thm:dim0}, we get the following result.

\begin{theorem}\label{thm:symbolic}
 For any subshift $(X, T)$ on $r$ letters, one has $h^*(X)=\log \ell$ for some $\ell \in \{1,\ldots,r\}$. Moreover, one of the following alternatives holds.
 \begin{itemize}
     \item[(a)] If $\ell =r$, then $p_{X}^*(n)=r^n$ for all $n \in \mathbb{N}$.
     \item[(b)] If $\ell\le r-1$, then there exists a constant $c>0$ such that $\ell^n \le p_{X}^*(n) \le n^c \ell^n$ for all $n \ge 2$.
 \end{itemize}
\end{theorem}
We remark that for $\ell \le r-1$, a much weaker bound of the form $p_{X}^*(n) \le n^{O(\log n)} \ell^n$ was obtained by Huang and Ye \cite[Theorem 5.5]{HY09}.

For the rest of this section, we give a more explicit expression for $p^*_{X}(n)$. Huang and Ye \cite{HY09} claimed, without a proof, that
\begin{equation}\label{eq:complexity-subshift}
p_X^*(n)=p^*_{X,\,\mU_0}(n)=\max _{0\le s_1<\ldots <s_n} \#\left\{x_{1+s_1}\ldots x_{1+s_n}\colon x\in X\right\}.
\end{equation}
For the reader's convenience, we provide a proof of this simple fact. Let $s_1<\ldots<s_n$ be any sequence of $n$ nonnegative integers. Since $\mU_0$ is a clopen partition of $X$, it follows from \cref{lem:clopen-family} that
\begin{align*}
N\big(\bigvee_{i\in [n]} T^{-s_i} \mU\big)&=\# \Big\{f\in [r]^{[n]}: \bigcap_{i\in [n]}T^{-s_i}U_{f(i)} \neq \emptyset\Big\}\\
&=\# \Big\{f\in [r]^{[n]}: \text{ there is $x\in X$ with } x\in \bigcap_{i\in [n]}T^{-s_i}U_{f(i)}\Big\}\\
&=\# \Big\{f\in [r]^{[n]}: \text{ there is $x\in X$ with } T^{s_i}x\in U_{f(i)} \text{ for every $1\le i\le n$}\Big\}\\
&=\# \Big\{f\in [r]^{[n]}: \text{ there is $x\in X$ with } x_{1+s_1}=f(1),\ldots, x_{1+s_n}=f(n)\Big\}\\
&=\#\left\{x_{1+s_1}\ldots x_{1+s_n}\colon x\in X\right\}.
\end{align*}
Taking the maximum over all sequences $s_1<\ldots<s_n$ yields \eqref{eq:complexity-subshift}.

In a special case when $X$ is the closure of the orbit of a word $a\in \Omega_r$ under the shift map $T$, one can further simplify \eqref{eq:complexity-subshift}. Indeed, for every $x\in X=\overline{\{a,Ta,T^2a,\ldots\}}$ and for every $n_0\in \NN$, there exists $y\in \{a,Ta,T^2a,\ldots\}$ such that $d(x,y)<2^{-n_0}$. Suppose $y=T^{m-1}a$ for some $m\in \NN$. Then the condition $d(x,y)<2^{-n_0}$ forces $x_i=y_i=a_{m+i}$ for all $i \le n_0$. Together with \eqref{eq:complexity-subshift}, this implies
\[
p^*_X(n)=\max_{0\le s_1<\ldots<s_n}\#\left\{a_{m+s_1}\ldots a_{m+s_n}\colon m\in \NN\right\}.
\]
The right-hand side of the equation is also called the {\bf maximal pattern complexity} of $a$. Actually, this concept was introduced by Kamae and Zamboni \cite{KZ02a,KZ02b}, and was the inspiration behind the work of Huang and Ye \cite{HY09} on the maximal pattern entropy.

\section{Concluding remarks}\label{sec:conclusion}

In this paper we study two variants of the VC dimension and their connections among dynamical systems, combinatorics and theoretical computer science.
One intriguing question that deserves further investigation is the tightness of the bounds in \cref{thm:comb} for $r\ge 3$. \cref{thm:lower-bound} tells us that these bounds are essentially tight for $r=2$. Some of the arguments in the proof of \cref{thm:lower-bound} do generalise to larger $r$.
Let us recall some notions from hypergraph theory. Let $G=(V,E)$ be an $r$-graph. The $r$-partition number of $G$, denoted by $f_r(G)$, is the minimum number of complete $r$-partite $r$-graphs needed to partition the edge set of $G$. The chromatic number $\chi(G)$ of $G$ is the minimum $k$ for which there exists a coloring $c\colon V\rightarrow [k]$ such that every edge $e\in E$ contains two vertices $u,v$ with $c(u)\ne c(v)$.
One can easily extend \cref{lem:graph-to-class} to $r$-graphs as follows.
\begin{proposition}
For every $r$-graph $G$ with $f_r(G)=n$, there exists a partial concept class $\mH \subseteq \{1,\ldots,r,\star\}^{[n]}$ with $\VC(\mH) \le 1$ and $C(\mH) \ge \chi(G)$.
\end{proposition}
What is missing is an analogue of \cref{thm:separation} for $r$-graphs.

\begin{question}[A hypergraph Alon--Saks--Seymour problem]
Let $r\ge 3$. For every $n$, is there an $r$-graph $G$ such that $f_r(G)=n$ and $\chi(G)\ge n^{c_r\log n}$, where $c_r>0$ is a constant depending only on $r$?
\end{question}

\subsection*{Acknowledgement}
G. G. was supported in part by the National Key Research and Development Program of China 2023YFA1010201, National Natural Science Foundation of China grant 12401448, and Natural Science Foundation of Fujian Province 2024J08030.
J. M. was supported in part by the National Key Research and Development Program of China 2023YFA1010201, National Natural Science Foundation of China grant 12125106, and Innovation Program for
Quantum Science and Technology 2021ZD0302902.
T. T. was supported by the National Key Research and Development Program of China 2023YFA1010201 and Excellent Young Talents Program (Overseas) of the National Natural Science Foundation of China.

\appendix

\section{Proof of \texorpdfstring{\cref{thm:lower-bound}}{}}\label{sec:lower-bound}
In this section we present a proof of \cref{thm:lower-bound}, due to Alon, Hanneke, Holzman and Moran \cite{AHHM21}. The proof exploits a recent breakthrough in communication complexity and graph theory, namely \cref{thm:separation} below, which provides a near-optimal solution to the Alon--Saks--Seymour problem in graph theory (for background on this problem, see the survey by Bousquet, Lagoutte and Thomass\'e \cite{BLT14}). Let $G=(V,E)$ be a simple graph. Recall that the chromatic number of $G$, denoted by $\chi(G)$, is the minimum $k$ for which there exists a labelling $c\colon V\rightarrow [k]$ such that every edge $\{u, v\}\in E$ satisfies $c(u)\ne c(v)$. The {\bf biclique partition number} of $G$, denoted $\bp(G)$, is the minimum number of bicliques (i.e. complete bipartite graphs) needed to partition the edges of $G$. The following result follows from a recent line of breakthroughs by G\"o\"os \cite{Goo15}; G\"o\"os, Lovett, Meka, Watson and Zuckerman \cite{GLMWZ16}; Balodis,
Ben-David, G\"o\"os, Jain and Kothari \cite{BBGJK21}:

\begin{theorem}[{\cite{BBGJK21}}]\label{thm:separation}
For every positive integer $n$ there exists a graph $G$ with $\bp(G)=n$ and
\[
\chi(G)\ge n^{(\log n)^{1-o(1)}},
\]
where the term $o(1)$ tends to zero as $n$ goes to infinity.
\end{theorem}

The following result allows us to use the graph $G$ promised by \cref{thm:separation}
to construct a partial concept class $\mH$ with small VC dimension and large covering number.

\begin{lemma}\label{lem:graph-to-class}
For every graph $G$ with $\bp(G)=n$, there exists a partial concept class $\mH \subseteq \{1,2,\star\}^{[n]}$ with $\VC(\mH) \le 1$ and $C(\mH) \ge \chi(G)$.
\end{lemma}

We remark that our proof also gives that $|\mH|=|V(G)|.$
Before proving \cref{lem:graph-to-class}, let us deduce \cref{thm:lower-bound} from it.

\begin{proof}[Proof of \cref{thm:lower-bound} assuming \cref{lem:graph-to-class}]
Let $G$ be the graph given by \cref{thm:separation}. Then we have $\bp(G)=n$ and $\chi(G)\ge n^{(\log n)^{1-o(1)}}$. By \cref{lem:graph-to-class}, there exists a partial concept class $\mH \subseteq \{1,2,\star\}^{[n]}$ with $\VC(\mH) \le 1$ and $C(\mH) \ge \chi(G) \ge n^{(\log n)^{1-o(1)}}$, proving \cref{thm:lower-bound}.
\end{proof}

To complete the proof of \cref{thm:lower-bound}, it remains to prove \cref{lem:graph-to-class}, the task we now pursue.

\begin{proof}[Proof of \cref{lem:graph-to-class}]
Suppose we have a partition of $E(G)$ as disjoint union of  $\mB(L_i, R_i)$ for $i\in [n]$, where $\mB(L_i
, R_i)$ denotes the edge set of the complete bipartite graph with parts $L_i$ and $R_i$. For each $v\in V(G)$, let $h_v$ be a partial function in $\{1,2,\star\}^{[n]}$ given by
\[
h_v(i)=
\begin{cases} 1 & \text { if } v \in L_i \\
2 & \text { if } v \in R_i\\
\star & \text { otherwise. }
\end{cases}
\]
Set $\mH=\{h_v\colon v\in V(G)\}$. We have to show that the partial concept class $\mH\subseteq \{1,2,\star\}^{[n]}$ satisfies $\VC(\mH)\le 1$ and $C(\mH) \ge \chi(G)$.

Suppose for the contrary that $\VC(\mH) \ge 2$. Then there must exist two distinct coordinates $i,j \in [n]$ such that $\{i,j\}$ is shattered by $\mH$.
In particular, we can find two vertices $u,v \in V(G)$ with
$(h_u(i),h_u(j))=(1,1)$ and $(h_v(i),h_v(j))=(2,2)$. From the definitions of $h_u$ and $h_v$, we get $u\in L_i \cap L_j$ and $v\in R_i \cap R_j$. Hence $\{u,v\}$ is covered by both $\mB(L_i,R_i)$ and $\mB(L_j,R_j)$, which contradicts the assumption that $\cup_{i=1}^n\mB(L_i,R_i)$ is an edge partition of $G$.

It remains to show that $C(\mH) \ge \chi(G)$. Indeed, from the definition of $C(\mH)$, there exists a family $\mC \subseteq \{1,2\}^{[n]}$ that satisfies
\begin{itemize}
    \item[(i)] $|\mC|=C(\mH)$;
    \item[(ii)] for each vertex $v\in V$ there is a total function $c_v\in \mC$ such that $h_v(i)\ne c_v(i)$ for all $i\in [n]$.
\end{itemize}
Assign to each vertex $v\in V$ the color $c_v \in \mC$. We claim that this is a proper coloring, and so $C(\mH)=|\mC| \ge \chi(G)$, as desired. Indeed, let $\{u,v\}$ be any edge in $G$. Since $\cup_{i=1}^n\mB(L_i,R_i)$ is an edge partition of $G$, $\{u,v\} \in \mB(L_i,R_i)$ for some $i\in [n]$. Let $u\in L_i$ and $v\in R_i$. By the definitions of $h_u$ and $h_v$, we thus obtain $h_u(i)=1$ and $h_v(i)=2$. It then follows from (ii) that $c_u(i)=2$ and $c_v(i)=1$. We conclude that $u$ and $v$ are assigned different colors, completing our proof.
\end{proof}

\section{Proof of \texorpdfstring{\cref{thm:HY-general}}{}}\label{sec:topo}
In this section we present Huang and Ye's proof \cite{HY09} of \cref{thm:HY-general}, with several simplifications. Let $(X,T)$ be a TDS. For a finite open cover $\mU$ of $X$, define $L(\mU)=\limsup\limits_{n\rightarrow+\infty} \frac{\log p^{*}_{X,\,\mU}(n)}{\log ^2(n+1)}$. Then one can restate \cref{thm:HY-general} as follows.

\begin{theorem}\label{thm:HY-restated}
If $(X,T)$ is a null TDS, then for every finite open cover $\mU$
\[
L(\mU)=O(1).
\]
\end{theorem}

Throughout the section, $A^c$ denotes the complement $X\setminus A$ of $A$.

\begin{proof}[Proof of \cref{thm:HY-restated}]
Suppose for the contrary that $L(\mU)=\infty$. By \cref{lem:Blanchard} below, there is an open cover $\mV=\{V_1,V_2\}$ of size $2$ such that  $L(\mathcal{V})=\infty$. Then, for each $d\in \NN$, there exists a finite set $S \subset \NN\cup \{0\}$ with $N\big(\bigvee_{i\in S} T^{-i} \mV\big)> (|S|+1)^{4d\log (|S|+1)}$. For each $x\in X$,
let $h_x$ be a partial function in $\{1,2,\star\}^{S}$ defined as
\[
h_x(i)=
\begin{cases}
1 & \text { if } x\in T^{-i}(V_1^c) \\
2 & \text { if } x \in T^{-i}(V_2^c) \\
\star & \text { otherwise}.
\end{cases}
\]
Since $\{V_1,V_2\}$ is a cover of $X$, $V_1^c$ and $V_2^c$ are disjoint, and hence $h_x$ is well-defined. Consider the partial concept class $\mH:=\{h_x\colon x\in X\}$.

\begin{claim}\label{claim:topo-comb}
$C(\mH)=N\big(\bigvee_{i\in S} T^{-i} \mV\big)$.
\end{claim}

\begin{proof}
Let $\mF \subseteq \{1,2\}^S$ be a total class. We can infer from the definitions of $\mH$ and $h_x$ that
\begin{align*}
\text{$\mF$ is a net of $\mH$} \quad &\Longleftrightarrow \quad \text{for each $h_x\in \mH$ there is $f\in \mF$ such that $h_x(i) \ne f(i)$ for all $i\in S$} \\
&\Longleftrightarrow \quad \text{for each $x\in X$ there is $f\in \mF$ such that $x\notin T^{-i}(V_{f(i)}^c)$ for all $i\in S$}\\
&\Longleftrightarrow \quad \text{for each $x\in X$ there is $f\in \mF$ such that $x\in T^{-i}V_{f(i)}$ for all $i\in S$}\\
&\Longleftrightarrow \quad \text{for each $x\in X$ there is $f\in \mF$ such that $x\in \bigcap_{i\in S}T^{-i}V_{f(i)}$}\\
&\Longleftrightarrow \quad \Big\{\bigcap_{i\in S}T^{-i}V_{f(i)}\colon f\in \mF\Big\} \text{ is a cover of $X$.}
\end{align*}
Letting $\mF$ be a net of minimum size of $\mH$, this yields the lower bound
\[
C(\mH)=|\mF|\ge |\big\{\bigcap_{i\in S}T^{-i}V_{f(i)}\colon f\in \mF\big\}| \ge N\big(\bigvee_{i\in S} T^{-i} \mV\big).
\]
For the upper bound, let $\mW$ be a minimum subcover of $\bigvee_{i\in S}T^{-i}\mV$. Then $\mW$ can be written as $\mW=\big\{\bigcap_{i\in S}T^{-i}V_{f(i)}\colon f\in \mF\big\}$, where $\mF \subseteq \{1,2\}^S$ is a total class with $|\mF|=|\mW|$. As $\mW$ is a cover of $X$, we find that $\mF$ is a net of $\mH$. Thus $C(\mH) \le |\mF|=|\mW|=N\big(\bigvee_{i\in S} T^{-i} \mV\big)$. We are done.
\end{proof}

From \cref{claim:topo-comb} we obtain $C(\mH)=N\big(\bigvee_{i\in S} T^{-i} \mV\big)\ge (|S|+1)^{4d\log (|S|+1)}$. Thus, by
the remark after \cref{thm:comb}, there exists a size-$d$ subset $W\subseteq S$ with $\restr{\mH}{W} \supseteq\{1,2\}^{W}$. Let $x\in X$ be an element such that $\restr{h_x}{W}$ lies in $\{1,2\}^W$. Suppose $x$ is contained in $\bigcap_{i\in W} T^{-i} V_{f(i)}$ for some $f\in \{1,2\}^W$, then we must have $h_x(i)\ne f(i)$ for all $i\in W$, and so $f$ is uniquely determined by $\restr{h_x}{W}$. Therefore, we have
\[
N\big(\bigvee_{i\in W}T^{-i}\mV\big) \ge |\restr{\mH}{W}\cap \{1,2\}^W|=2^{|W|}.
\]
Letting $|W|=d \rightarrow \infty$ yields $h^*_{\topo}(T,\mV) \ge 1>0$, a contradiction.
\end{proof}

The rest of this section is devoted to establishing the following lemma that was used in the proof \cref{thm:HY-restated}.

\begin{lemma}\label{lem:Blanchard}
Let $(X,T)$ be a TDS. Suppose that $\mU$ is a finite open cover of $X$ with
$L(\mU)=+\infty$. Then there is an open cover $\mathcal{V}=\{V_1,V_2\}$ of size two with $L(\mV)=+\infty$.
\end{lemma}

As in \cite{HY09}, we follow the arguments of Blanchard \cite{Bla93}. We shall make use of basic properties of the function $L$.

\begin{proposition}\label{prop:entropy-properties}
Let $(X,T)$ be a TDS. Then the following properties hold.
\begin{itemize}
    \item[\rm(i)] ({\bf Monotone}) If $\mV$ is a refinement of $\mU$, then $L(\mU) \leq L(\mV)$.
    \item[\rm(ii)] ({\bf Subadditive}) If $\mU,\mV$ are open covers of $X$, then $L(\mU\vee\mV) \le L(\mU)+L(\mV)$.
\end{itemize}

\end{proposition}
\begin{proof} For (i), consider any set $S$ of nonnegative integers. Since $\mV$ is a refinement of $\mU$, $\bigvee_{i\in S}T^{-i}\mV$ is a refinement of $\bigvee_{i\in S}T^{-i}\mU$, and so $N\big(\bigvee_{i\in S}T^{-i}\mU\big) \le  N\big(\bigvee_{i\in S}T^{-i}\mV\big)$. It follows that $p^*_{X,\, \mU}(n) \le p^*_{X,\, \mV}(n)$ for every $n\in \NN$, which implies $L(\mU)\le L(\mV)$.

For every $n\in \NN$ we have
\begin{align*}
p^*_{X,\,\mU\vee \mV}(n)=\max_{S} N\Big(\bigvee_{i\in S} T^{-i} (\mU\vee \mV)\Big)&=\max_{S} N\Big(\bigvee_{i\in S} (T^{-i}\mU \vee T^{-i}\mV)\Big)\\
&\leq \max_{S} N\Big(\bigvee_{i\in S} T^{-i}\mU \Big)\cdot N\Big(\bigvee_{i\in S} T^{-i}\mV \Big)\\
&\le \max_{S}N\Big(\bigvee_{i\in S} T^{-i}\mU \Big)\cdot\max_{T}N\Big(\bigvee_{i\in T} T^{-i}\mV \Big)\\
&=p^*_{X,\,\mU}(n)\cdot p^*_{X,\,\mV}(n),
\end{align*}
where the maximums are taken over all size-$n$ subsets $S,T \subset \NN \cup \{0\}$.
Taking the logarithm and then dividing by $\log^2(n+1)$, we obtain $L(\mU\vee \mV) \le L(\mU)+L(\mV)$.
\end{proof}

\begin{proof}[Proof of \cref{lem:Blanchard}]
Let $\mU=\{U_1,\ldots,U_k\}$. We first observe that $U_1\neq X$. Suppose otherwise that $U_1=X$, then for every subset $S \subseteq \NN\cup \{0\}$ we have $X\in \bigvee_{i\in S} T^{-i} \mU$, and so the open cover $\bigvee_{i\in S} T^{-i} \mU$ has a subcover of size $1$, which implies $L(\mU)=0$, a contradiction.

Given a subset $A\subseteq X$, we shall use $\diam(A)$ to denote the diameter of $A$. We now inductively construct a sequence of closed sets $(A_n)_{n\ge 0}$ with the following three properties
\begin{itemize}
    \item[\rm (P1)] $U_1^c=A_0 \supseteq A_1 \cdots$;
    \item[\rm (P2)] $\diam(A_n) \le 2^{-n}$ for every $n\ge 1$;
    \item[\rm (P3)] $L(\mU_n)=+\infty$ for every $n\geq 0$, where $\mU_n:=\{A_n^c,U_2,\ldots,U_k\}$.
\end{itemize}
Clearly, $A_0=U_1^c$ is a closed set that satisfies (P1)--(P3). Suppose that we have already constructed $A_n$.
 Since $A_n$ is a closed subset of the compact set $X$, $A_n$ is also compact, and so we can cover $A_n$ by a finite number of closed balls $B_1,\ldots,B_{\ell}$ of radius $2^{-n-2}$. For $1\le i \le \ell$, let $A_{n+1,i}=A_n\cap B_i$ and $\mU_{n+1,i}=\{A_{n+1,i}^c,U_2,\ldots,U_k\}$. Since $\mU_n=\{A_n^c,U_2,\ldots,U_k\}$ is an open cover of $X$ and $A_{n+1,i}^c$ is an open set containing $A_n^c$, we find that $\mU_{n+1,i}$ is also an open cover of $X$. The definition of $A_{n+1,i}$ gives $A_{n+1,i}^c=(A_n\cap B_i)^c=A_n^c\cup B_i^c$. Hence
 \[
     \bigcap_{i=1}^{\ell} A_{n+1,i}^c = \bigcap_{i=1}^{\ell} \big(A_n^c\cup B_i^c\big)=A_n^c\cup \big(A_n\cap \bigcap_{i=1}^{\ell}B_i^c\big)=A_n^c\cup \emptyset=A_n^c,
 \]
where the third identity holds since $A_n$ is covered by $B_1,\ldots,B_{\ell}$. It follows that every set in $\bigvee_{i=1}^{\ell} \mU_{n+1,i}$ is contained in some set in $\mU_n=\{A_n^c,U_2,\ldots,U_k\}$. In other words, the open cover $\bigvee_{i=1}^{\ell} \mU_{n+1,i}$ is a refinement of $\mU_n$. Using \cref{prop:entropy-properties}, we thus obtain
\[
+\infty=L(\mU_n
)\leq L\big(\bigvee_{i=1}^{\ell} \mU_{n+1,i}\big)\leq \sum_{i=1}^{\ell} L(\mU_{n+1,i}).
\]
So $L(\mU_{n+1,i})=+\infty$ for some $i\in [\ell]$.
Set $A_{n+1}=A_{n+1,i}$. As $A_{n+1}=A_n\cap B_i$, $A_{n+1}$ is a closed subset of $A_n$ of diameter $\diam(A_{n+1})\leq \diam(B_i) \leq 2\cdot 2^{-n-2}=2^{-n-1}$. Therefore, $A_{n+1}$ has the desired properties.

From (P3) and the observation at the beginning of the proof, we see that $A_n^c\neq X$, and so $A_n\neq \emptyset$. This, together with (P1), (P3) and the compactness of $X$, yields $\bigcap_{n\ge 0}A_n=\{x\}$ for some $x\in X$. Because $x\in A_0=U_1^c$ and $\mU=\{U_1,\ldots,U_k\}$ is a cover of $X$, there is $\ell \in \{2,\ldots,k\}$ with $x\in U_{\ell}$. Since $U_{\ell}$ is open, there exists $\eps>0$ such that $U_{\ell}$ contains an open ball $B(x,\eps)$ centered at $x$ of radius $\eps>0$. On the other hand, since $\bigcap_{n\geq 0}A_n=\{x\}$ and $\lim_{n\rightarrow \infty}\diam(A_n)=0$, one has $A_n \subseteq B(x,\eps)$ for $n$ sufficient large. For such an $n$, let $V_1=A_n^c$ and $V_2=U_{\ell}$. Since $V_1\cup V_2 \supseteq B(x,\eps)^c\cup B(x,\eps)=X$,  $\mathcal{V}=\{V_1,V_2\}$ is a subcover of $\mU_n$. Finally, we have
\[
L(\mV)\geq L(\mU_n)=+\infty.
\]
This completes our proof.
\end{proof}

\end{document}